\documentclass[reqno]{amsart}
\usepackage{cases}
\usepackage{latexsym}
\usepackage{amsmath}
\usepackage[arrow,matrix]{xy}
\usepackage{stmaryrd}
\usepackage{amsfonts}
\usepackage{amsmath,amssymb,amscd,bbm,amsthm,mathrsfs,dsfont}

\usepackage{fancyhdr}
\usepackage{amsxtra,ifthen}
\usepackage{verbatim}

 \usepackage{etoolbox}
 \usepackage[colorinlistoftodos]{todonotes}

\DeclareMathOperator{\JD}{JDer}
\numberwithin{equation}{section}

\theoremstyle{plain}
\newtheorem{theorem}{Theorem}[section]
\newtheorem{lemma}[theorem]{Lemma}

\newtheorem{corollary}[theorem]{Corollary}

\theoremstyle{definition}

\newtheorem{example}[theorem]{Example}

\newtheorem{remark}[theorem]{Remark}

\begin{document}

\title[ Jordan $*$-derivations of incidence algebras]
{ Jordan $*$-derivations of incidence algebras}

\author{Liuqing Yang}

\address{Liuqing Yang: Department of Mathematics, Soochow University, Suzhou, 215006,  P.R. China
}

 \email{20214007001@stu.suda.edu.cn}


\begin{abstract}
Let $X$ be a locally finite partially ordered set (poset), $K$ a field of characteristic not 2, and $I(X,K)$ the incidence algebra over $K$. In this paper, we prove that every Jordan $*$-derivation of $I(X,K)$ is an inner $*$-derivation and a transposed Jordan $*$-derivation. Moreover, we demonstrate the existence of Jordan $*$-derivations that are not $*$-derivations.
\end{abstract}

\maketitle

\section{Introduction}\label{xxsec1}
Throughout this paper we assume that $K$ is a field of characteristic not 2. Let $A$ be an associative $K$-algebra, 
an anti-automorphism $*: f\to f^*$ satisfying  $(f^* )^* = f$ is called an involution. 
A ring equipped with an involution is called a $*$-ring.
A $K$-linear map  $D$ of $A$  is called a \textsf{$*$-derivation}, if   
$$D(fg)=D(f)g^*+ f  D(g),$$
for all $f,g \in A$, it is called an \textsf{inner $*$-derivation}, if  for some $f\in A$ 
$$D(g)=fg^*- g f,$$
for all $g \in A$, and is called a \textsf{Jordan $*$-derivation}, if 
\begin{align*} D(f^2)=D(f)f^*+ f  D(f),
\end{align*} 
for all $f\in A$.  
It is clear that any $*$-derivation of $A$ is a Jordan $*$-derivation.
However, the converse statements are not always true.

The study of Jordan $*$-derivations has been driven by the problem of whether quasi-quadratic functionals can be represented by sesquilinear ones. It has been revealed that this question is closely related to the structural theory of Jordan $*$-derivations 
(see  \cite{17,23, 29,30} and references therein). Bre{\v{s}}ar and Vukman \cite{Bre} showed that if a unital $*$-ring with some conditions then every additive Jordan $*$-derivation  is inner $*$-derivation. In particular, every additive Jordan $*$-derivation on a unital complex $*$-algebra is inner $*$-derivation.
{\v{S}}emrl \cite{31} proved that every Jordan $*$-derivation of $B(H)$, the algebra of all bounded linear operators on a real Hilbert space $H$ with $dim RH > 1$, is inner $*$-derivation.
Let $R$ be a finite-dimensional central simple algebra with involution $*$,
which is not commutative, then any Jordan $*$-derivation of $R$ is inner  $*$-derivation(see \cite{F}).

The objective of this article is to investigate Jordan $*$-derivations on incidence algebras $I(X,K)$. Many authors have made important contributions to the related topics. 
Xiao \cite{Xiao15}  proved that Jordan derivations of incidence algebras with $X$ a locally finite set  are derivations, Khrypchenko \cite{Khr16} proved that Jordan derivations of finitary incidence algebras are derivations. More precisely, for the incidence algebra $I(X,R)$ with $X$ a locally finite pre-ordered set and $R$ a unital $2$-torsion free commutative ring,
it is proved that each Jordan derivation of $I(X,R)$ is a derivation and each Lie derivation of $I(X,R)$
is the sum of a derivation and a central-valued map , see \cite{Zhang17}. Since then, more and more Lie-type maps of incidence algebras
have been studied, see \cite{CH-X1, JiaXiao} and the references therein.

In this paper,  we aim to describe Jordan $*$-derivations of incidence algebras, as well as  motivate what drives our work: the relation between inner $*$-derivations and Jordan $*$-derivations.

\section{Preliminaries} 
In view of the results above, it seems reasonable to study the following.
\subsection{Incidence algebras}
We now recall the notion of incidence algebras \cite{SpDo}.
Let $(X,\leqslant)$ be a locally finite pre-ordered set (i.e., $\leqslant$ is a reflexive and transitive binary relation).
Here the local finiteness means for any
$x\leqslant y$ there are only finitely many elements $z\in X$ satisfying $x\leqslant z\leqslant y$.
Let $R$ be a commutative ring with identity. We say that an
incidence algebra is a certain ring of functions.  We set 
$$I(X,R)=\{f\colon X\times X\to R\,|\, f(x,y)=0,\,\text{if }\, x\not\leqslant y\}.$$ 
The functions are added point-wise. The product of the functions $f,g\in I(X,R)$ is given by the relation
$$
(fg)(x,y)=\sum_{x\leqslant z\leqslant y}f(x,z) g(z,y),
$$
for any $x,y\in X$. Since $X$ is a locally finite set, there are only  finite number of $z$ such that   $x\leqslant z\leqslant y$. For any $r\in R$ and $x,y\in X$, we also assume $(r f)(x,y)=r f(x,y)$. Therefore, $I(X,R)$ is an $R$-algebra, called  \textsf{incidence algebra}. The unity element $\delta$ of $I(X,R)$ is given by $\delta(x,y)=\delta_{xy}$,
where $\delta_{xy}\in \{0,1\}$ is the Kronecker delta. For all $x\leqslant y$ in $X$, let us define $e_{xy}$ to be
\begin{displaymath}
e_{xy}(u, v)= \left\{ \begin{array}{ll}
 1 & \textrm{$(u, v)=(x, y)$}; \\
 0 & \textrm{otherwise}.
  \end{array} \right.
\end{displaymath}
By the definition of convolution, $e_{xy}e_{uv}=\delta_{yu}e_{xv}$, for all $x\leqslant y,u\leqslant v$ in $X$,
follows immediately.

For all $f,g\in I(X,R)$, we denote $f\odot  g$ the \textsf{Hadamard product} or \textsf{Schur product} by 
$(f\odot g)(x,y)=f(x,y)g(x,y)$ \cite{SpDo}.

If $X$ is a finite set, then the ring $I(X,R)$ is isomorphic to a certain subring of the matrix ring $M(n,R)$,  in this condition,  $I(X,R)$ is often called \textsf{ring of structural matrices}.

\subsection{Involution}

Let $\mu\in I(X,K)$ be an invertible element, we denote by $\psi_\mu$ the \textsf{inner automorphism} of  $I(X,K)$ given by
\begin{align*}
\psi_\mu(f)=\mu f\mu^{-1}, ~\text{for~all}~f \in I(X,K).
\end{align*}
A nonzero element $\sigma\in I(X,K)$ is \textsf{multiplicative} if
\begin{align*}
\sigma(x,y)=\sigma(x,z)\sigma(z,y), ~\text{for~all}~x\leqslant z\leqslant y~\text{in~}  X,
\end{align*}
and
\begin{align*}
\sigma(x,x)=1, ~\text{for~all}~x \in X.
\end{align*}
A multiplicative element $\sigma$ determines the \textsf{multiplicative automorphism}
$M_{\sigma}$ by $$M_{\sigma}(f)=\sigma\odot f.$$

If $\lambda$ is an involution of $X$, i.e. $\lambda^2 (x)=x$ for all $x\in X$ and $\lambda(y)<\lambda(x)$ for all $x<y$ in $X$, then $\lambda$ induces an involution $\hat{\lambda}$ of
$I(X, K)$ given by $ \hat{\lambda}(f) (x, y) =f(\lambda(y), \lambda(x))$, for each $f\in I(X, K)$ and each pair $x, y\in X$.
\begin{theorem}\label{autin}\cite[The decomposition theorem]{Brusamarello}
Let $X$ be a locally finite poset and $K$ a field. Then every involution $\theta$ of $I(X,K)$ is of the form $\theta=\psi_{\mu}\circ M_{\sigma}\circ\hat{\lambda}$ for some invertible $\mu\in I(X,K)$, involution $\hat{\lambda}$ induced by an involution $\lambda$ of $X$ and multiplicative element $\sigma\in I(X,K)$.
\end{theorem}

\begin{remark}\label{autin1} 
If an involution $\phi$ of $I(X,K)$ is of the form $\phi=M_{\sigma}\circ\hat{\lambda}$, then 
for all $f\in I(X,K)$ and $x\leqslant y$ in $X$, we have
\begin{align*}
 \phi(f)(x,y)&=M_{\sigma}(\hat{\lambda}(f))(x,y)\\
 &=(\sigma \odot \hat{\lambda}(f))(x,y)\\
 &=\sigma(x,y)  \hat{\lambda}(f)(x,y)\\
 &=\sigma(x,y)f(\lambda(y),\lambda(x)).
\end{align*}
By the definition of involution $\lambda$, we have $\lambda(y)\leqslant\lambda(x)$ and $\hat{\lambda}  (f)(\lambda(y),\lambda(x))=f(x,y)$, then 
$\phi(f)(\lambda(y),\lambda(x))=\sigma(\lambda(y),\lambda(x))f(x,y)$. This means
$$\phi(e_{xy})=\sigma(\lambda(y),\lambda(x))e_{\lambda(y)\lambda(x)}.$$ 
\end{remark}

\subsection{Inner $*$-derivations}
In what follows $X$ is a locally finite partially ordered set(poset for short), $K$  a field of characteristic not 2 and $I(X,K)$ an incidence algebra over $K$.  

For any fixed $f\in I(X,K)$,  we denote the $K$-linear map $\Delta_{f}^{\theta}$ be an inner $*$-derivation under the involution $\theta$ by $\Delta_{f}^{\theta}(g)=f\theta(g)-gf$, for all $g\in A$.
$R_\mu$ be a right multiplication operator defined by $R_\mu(g)=g\mu$, for all $g\in A$.
\begin{lemma}\label{c2}
Let $\theta$ be an involution of $A$ having the form $\theta=\psi_{\mu}\circ\phi$, $\psi_{\mu }$ an inner automorphism and $\phi$ an involution. Then  $\Delta_{f\mu}^{\phi}= R_\mu\circ \Delta_{f}^{\theta}$.
\end{lemma}
\begin{proof}
For all $g\in A$, we have
 \begin{align*} 
R_\mu\circ \Delta_{f}^{\theta}(g)&=f \theta(g)\mu-gf\mu\\
&=f (\psi_{\mu}\circ\phi)(g)\mu-gf\mu\\
&=f\mu  \phi(g) -gf\mu\\
&=\Delta_{f\mu}^{\phi}(g).
 \end{align*}
 Therefore, $\Delta_{f\mu}^{\phi}= R_\mu\circ \Delta_{f}^{\theta}$.
\end{proof}
Due to Theorem \ref{autin} and Lemma \ref{c2},  we only focus our attention on involution $\phi$ of $I(X,K)$ is of the form $\phi=M_{\sigma}\circ\hat{\lambda}$, and we will use Remark \ref{autin1} frequently.

\begin{lemma}\label{cc2}
Let $\Delta_{f}^{\phi}$ be an $I(X,K)$ be an inner $*$-derivation, then  $\Delta_{f}^{\phi}(e_{xx})=0$ for all $x\in X$ if and only if
$\Delta_{f}^{\phi}(e_{xy})=f(y,\lambda(y))\sigma(\lambda(y),\lambda(x))e_{y\lambda(x)}-f(y,\lambda(y))e_{x\lambda(y)}$.
\end{lemma}
\begin{proof}
The "if" part is obvious, we only need to show the "only if" part. 
For all $x<y \in X$, by assumption, we have $\Delta_{f}^{\phi}(e_{yy})=0$, and
 \begin{align*} 
0=\Delta_{f}^{\phi}(e_{yy})=fe_{\lambda(y)\lambda(y)}-e_{yy}f.
 \end{align*}
 Equating the coefficient of $e_{u\lambda(y)}$ and $e_{yv}$ with $u\neq y, v\neq \lambda(y)$ in $X$, we have
\begin{align}\label{ccq1} f(u,\lambda(y))=0, f(y,v)=0, \end{align}
 respectively. And for all $x<y \in X$,
  \begin{align*} 
\Delta_{f}^{\phi}(e_{xy})&=f\sigma(\lambda(y),\lambda(x))e_{\lambda(y)\lambda(x)}-e_{xy}f\\
  &= \sum_{u\leqslant \lambda(y)}f(u, \lambda(y))e_{u\lambda(y)}\sigma(\lambda(y),\lambda(x))e_{\lambda(y)\lambda(x)}-
  e_{xy}\sum_{y\leqslant v}f(y,v)e_{yv}\\
  &=\sum_{u\leqslant \lambda(y)}f(u, \lambda(y))\sigma(\lambda(y),\lambda(x))e_{u\lambda(x)}-
  \sum_{y\leqslant v}f(y,v)e_{xv}\\
      &=f(y,\lambda(y))\sigma(\lambda(y),\lambda(x))e_{y\lambda(x)}-f(y,\lambda(y))e_{x\lambda(y)}.
 \end{align*}Where the last equation follow from (\ref{ccq1}).
\end{proof}


\section{Jordan $*$-derivations}
In what follows $K$ denotes  a field of characteristic not 2, $X$ an arbitrary pre-ordered set,  $D$   a   Jordan $*$-derivation of $I(X,K)$.
We consider the problem of describing the form of Jordan $*$-derivations of the incidence algebra. 

\begin{example}\label{exam}
Let $X=\{1,2\}$ such that $1<2$, and $I(X,K)$ be the incidence algebra, in this case, $I(X,K)$ is isomorphic to the $2\times 2$ upper triangular matrix algebra. Let $*$ be a $K$-linear map  of $I(X,K)$ defined by
$$\begin{bmatrix}
a_{11}&a_{12}\\
0&a_{22}
\end{bmatrix}^*=\begin{bmatrix}
a_{22}&a_{12}\\
0&a_{11}
\end{bmatrix}.$$
It satisfies \(*\circ*=\text{id}\) (the identity mapping) and  the multiplication  
$$\left(\begin{bmatrix}
a_{11}&a_{12}\\
0&a_{22}
\end{bmatrix}\begin{bmatrix}
b_{11}&b_{12}\\
0&b_{22}
\end{bmatrix}\right)^*=\begin{bmatrix}
b_{11}&b_{12}\\
0&b_{22}
\end{bmatrix}^*\begin{bmatrix}
a_{11}&a_{12}\\
0&a_{22}
\end{bmatrix}^*.$$
From this, we can conclude that $*$ is an involution.

Next, we define a $K$-map $D$ by $$D\left(\begin{bmatrix}a&b\\0&c\end{bmatrix}\right)=\begin{bmatrix}b&0\\0&-b\end{bmatrix}.$$
It can be easily obtained through calculation that 
\begin{align*}
D\left(\begin{bmatrix}a&b\\0&c\end{bmatrix}^2\right)=
D\left(\begin{bmatrix}a&b\\0&c\end{bmatrix}\right)
\begin{bmatrix}a&b\\0&c\end{bmatrix}^*+\begin{bmatrix}a&b\\0&c\end{bmatrix}D\left(\begin{bmatrix}a&b\\0&c\end{bmatrix}\right),
\end{align*}
this means that $D$ is a Jordan $*$-derivation.

\begin{align*}
D\left(\begin{bmatrix}a_{11}&a_{12}\\0&a_{22}\end{bmatrix}\begin{bmatrix}b_{11}&b_{12}\\0&b_{22}\end{bmatrix}\right)
&=\begin{bmatrix}a_{11}b_{12}+a_{12}b_{22}&0\\0&-a_{11}b_{12}-a_{12}b_{22}\end{bmatrix}
\end{align*}
\begin{align*}
&D\left(\begin{bmatrix}a_{11}&a_{12}\\0&a_{22}\end{bmatrix}\right)
\begin{bmatrix}b_{11}&b_{12}\\0&b_{22}\end{bmatrix}^*
+\begin{bmatrix}a_{11}&a_{12}\\0&a_{22}\end{bmatrix}
D\left(\begin{bmatrix}b_{11}&b_{12}\\0&b_{22}\end{bmatrix}\right)\\
&=\begin{bmatrix}a_{12}b_{22}+b_{12}a_{11}&0\\0&-a_{12}b_{11}-a_{22}b_{12}\end{bmatrix}
\end{align*}
This means that $D$ is not a \(*\)-derivation.
\end{example}
Example \ref{exam} shows that, for incidence algebra, \(*\)-derivation is proper subset of Jordan $*$-derivation. Next we consider the relation between Jordan \(*\)-derivation and inner $*$-derivation.

Let  $A$ be a $K$-algebra, the following conclusions hold true. 
\begin{lemma}\label{BRESAR}
 Let $A$ be a 2-torsion free $*$-ring,  $D$ be a   Jordan $*$-derivation. Then, for all $f,g, h\in A$,
 $$D(fg+gf)=D(f)g^*+D(g)f^*+fD(g)+gD(f).$$  
\end{lemma}

\begin{lemma}\label{1}
Let $\theta$ be an involution of $A$ having the form $\theta=\psi_{\mu }\circ\phi$, $\psi_{\mu }$ an inner automorphism and $\phi$ an involution, $D$ a Jordan $*$-derivation under the involution $\theta$. Then
$d:A\to A$ given by $D_0(f)=D(f)\mu $ is a   Jordan $*$-derivation under the involution $\phi$.
\end{lemma}
\begin{proof}
For all $f\in A$,  
$$D(f^2)=D(f)\theta(f)+ f D(f)=D(f)\mu \phi(f)\mu ^{-1}+ f D(f).$$
 Multiplying by $\mu$ from the right side, we have
 \begin{align*}
  D(f^2)\mu
= D(f) (\mu \phi(f)\mu ^{-1} )\mu +f D(f)\mu 
=  D(f)\mu\phi(f) +f  D(f)\mu ,
\end{align*}
and then $D_0(f^2)=D_0(f)\phi(f) +f D_0(f)$ for all  $f\in A$. Hance  $D_0$  is a   Jordan $*$-derivation under the involution $\phi$.
\end{proof}

We denote all the Jordan $*$-derivations under the involution $\phi$ of $A$ by $\JD^{\phi}(A)$.

\begin{corollary}\label{c1}
Let $\theta$ be an involution of $A$ having the form $\theta=\psi_{u}\circ\phi$, $\psi_{\mu }$ an inner automorphism and $\phi$ an involution. Then  $\JD^{\phi}(A)= \JD^{\theta}(A)\cdot\mu$.
\end{corollary}

Next, we only consider involution $\phi$ of $I(X,K)$  of the form $\phi=M_{\sigma}\circ\hat{\lambda}$ for some involution  $\lambda$ of $X$ and multiplicative element $\sigma\in I(X,K)$.

\begin{lemma}\label{2}
Let $\phi$ be an involution  of $I(X,K)$,  $D\in \JD^{\phi}$ a   Jordan $*$-derivation and define $f\in I(X,K)$ by $f(x,y)=D(e_{xx})(x,y)$ for all $x\leqslant y$. Then there exist an inner  $*$-derivation $\Delta_{f}^{\phi}$, such that $D(e_{xx})=\Delta_{f}^{\phi}(e_{xx})$ for all $x\in X$.
\end{lemma}
\begin{proof}
For all $x\in X$,
\begin{align}\label{eqj4}
\Delta_{f}^{\phi}(e_{xx})=f\phi (e_{xx})-e_{xx}f.
\end{align}
According to our assumption, we have
\begin{align*}
 e_{xx}f(u,v)=\begin{cases}
0, &if~u\neq x
\\
f(x,v), &if~u= x
\end{cases}
=\begin{cases}
0, &if~u\neq x
\\
D(e_{xx})(x,v), &if~u= x
\end{cases}=e_{xx}D(e_{xx})(u,v),
\end{align*}
for all $u\leqslant v$ in $X$. Therefore
\begin{align}\label{eqj5}
e_{xx}f=e_{xx}D(e_{xx}),  ~\text{for~all}~x \in X.
\end{align}
Replacing $f$ with $e_{xx}$ in the definition of Jordan $*$-derivations and multiplying $e_{xx}$ on the left side, we have
\begin{align}\label{eqj6}
e_{xx} D(e_{xx})\phi(e_{xx})=0,  ~\text{for~all}~x \in X.
\end{align}
Thus by (\ref{eqj5}) and (\ref{eqj6}) we have
\begin{align}\label{eqj7}
e_{xx} f\phi(e_{xx})=0,  ~\text{for~all}~x \in X.
\end{align}
Next, we prove $f\phi(e_{xx})+D(e_{xx})\phi(e_{xx})=0$. If $x\leqslant v$ in $X$, then by (\ref{eqj6}) and (\ref{eqj7})
\begin{align}\label{eqj8}
(f\phi(e_{xx})+D(e_{xx})\phi(e_{xx}))(x,v)=e_{xx}(f\phi(e_{xx})+D(e_{xx})\phi(e_{xx}))(x,v)=0.
\end{align}
If $x \neq u\leqslant v$ in $X$,
\begin{align}
&(f\phi(e_{xx})+D(e_{xx})\phi(e_{xx}))(u,v)e_{uv}\nonumber\\
=&e_{uu}f\phi(e_{xx})e_{vv}+e_{uu}D(e_{xx})\phi(e_{xx})e_{vv}\nonumber\\
=& (e_{uu}f\phi(e_{xx})+e_{uu} D(e_{xx})\phi(e_{xx}))e_{vv}.\label{eqj9}
\end{align}
Since $u\neq x$, $e_{uu}e_{xx}=e_{xx}e_{uu}=0$ and
$$0=D(e_{uu}e_{xx}+e_{xx}e_{uu})=
D(e_{xx})\phi(e_{uu})+D(e_{uu})\phi(e_{xx})+e_{xx}D(e_{uu})+e_{uu}D(e_{xx}).$$
Multiplying $\phi(e_{xx})$ from the right side and $e_{uu}$ on the left side,
\begin{align*}
0&=e_{uu}(D(e_{xx})\phi(e_{uu})+D(e_{uu})\phi(e_{xx})+e_{xx}D(e_{uu})+e_{uu}D(e_{xx}))\phi(e_{xx})\\
&=e_{uu}D(e_{uu})\phi(e_{xx})+e_{uu}D(e_{xx})\phi(e_{xx}).
\end{align*}
Thus, by (\ref{eqj5}) and (\ref{eqj9}), $(\phi(e_{xx})D(e_{xx})+\phi(e_{xx})f)(u,v)=0 $, for all  $u\leqslant v\neq x$ in $X$, and then by (\ref{eqj8}) we have
$$\phi(e_{xx})D(e_{xx})+\phi(e_{xx})f =0.$$
Together with (\ref{eqj5}), we have $\phi(e_{xx})f+\phi(e_{xx})D(e_{xx})+D(e_{xx})e_{xx}-fe_{xx}=0$, this means $D(e_{xx})=\Delta_{f}^{\phi}(e_{xx})$ for all $x\in X$.
\end{proof}

A Jordan $*$-derivation $D$  with
\begin{align*}
   D(e_{xy})=\begin{cases}
\gamma_{xy}e_{x\lambda(y)}-\gamma_{xy}\sigma(\lambda(y), \lambda(x)) e_{y\lambda(x)}, &if~x<\lambda(y)~and~y<\lambda(x)
\\
0, &otherwise,
\end{cases}
\end{align*}
where $\gamma_{xy}\in K$ such that $\gamma_{xy}=\gamma_{zy}$ if $x<y, z<y$ in $X$, then $D$ is called a \textsf{transposed Jordan $*$-derivation}. And the  Jordan $*$-derivation defined in Example \ref{exam} is a transposed Jordan $*$-derivation.

\begin{lemma}\label{3}
Let    $D\in \JD^{\phi}$ be a Jordan $*$-derivation, then  $D(e_{xx})=0$  for all $x\in X$
if and only if $D$ is a transposed Jordan $*$-derivation.
\end{lemma}
\begin{proof}
 Let $D$ be a transposed Jordan $*$-derivation, if $x<\lambda(x)$
 we have $D(e_{xx})=\gamma_{xx}e_{x\lambda(x)}-\gamma_{xx}\sigma(\lambda(x), \lambda(x)) e_{x\lambda(x)}=0$, then 
 $D(e_{xx})=0$  for all $x\in X$. Next, we consider the only if part.

For all  $x < y$  in  $X$, by  $D(e_{ii})=0$  for all $i\in X$, we have 
 \begin{align*} 
D(e_{xy}) &= D(e_{xx}e_{xy}+e_{xy}e_{xx})=  D(e_{xy})\phi(e_{xx})  + e_{xx} D(e_{xy})\\
D(e_{xy}) &= D(e_{yy}e_{xy}+e_{xy}e_{yy})= e_{yy} D(e_{xy}) +  D(e_{xy}) \phi(e_{yy}),
 \end{align*}
This implies that
 \begin{align}\label{eqlem1}
D(e_{xy}) = e_{xx} D(e_{xy}) \phi(e_{yy}) + e_{yy} D(e_{xy}) \phi(e_{xx}), \text{~for~all~}  x < y.
 \end{align}

For all  $x < y$ in  $X$, we have $e_{xy} e_{xy}=0$ and then
\begin{align} \label{eqlem5}
 0 = D(e_{xy} e_{xy}) = D(e_{xy}) \phi(e_{xy}) + e_{xy} D(e_{xy}).
\end{align}
According to Remark \ref{autin1} we have $\phi(e_{xy})=\sigma(\lambda(y), \lambda(x))e_{\lambda(y)\lambda(x)}$,  by (\ref{eqlem5}) and (\ref{eqlem1})
we have $e_{xx}D(e_{xy}) \phi(e_{xy}) + e_{xy} D(e_{xy})\phi(e_{xx})=0$ and
 \begin{align}\label{eqlem3}
 D(e_{xy}) (y, \lambda(x)) = - \sigma(\lambda(y), \lambda(x)) D(e_{xy})(x, \lambda(y)). 
\end{align}
Denote  $\gamma_{xy} = D(e_{xy})(x, \lambda(y))$, for all $x<y$.
Then by (\ref{eqlem3})  and (\ref{eqlem1}), we have
 \begin{align}
 D(e_{xy}) &= e_{xx} D(e_{xy}) e_{\lambda(x)\lambda(y)} + e_{yy} D(e_{xy}) e_{\lambda(x)\lambda(x)} \nonumber \\
 &= \gamma_{xy} e_{x \lambda(y)} - \gamma_{xy} \sigma(\lambda(y), \lambda(x)) e_{y \lambda(x)} \label{eqlem4}
 \end{align}

 For all  $x < z$  and  $y < z$ in  $X$. By $e_{xz}e_{yz} + e_{yz} e_{xz} = 0$ and  (\ref{eqlem1}), we obtain 
 \begin{align}
0 &=  D(e_{xz}  e_{yz} + e_{yz} e_{xz}) \nonumber \\
 &=  D(e_{xz}) \phi(e_{yz}) + e_{xz} D(e_{yz}) + D(e_{yz}) \phi(e_{xz}) + e_{yz} D(e_{xz})\nonumber \\
 &=e_{xx} D(e_{xz}) \phi(e_{yz}) + e_{xz} D(e_{yz}) \phi(e_{xx}). \label{eqlem2}
 \end{align}
 By (\ref{eqlem2}) and (\ref{eqlem4}), we have
 \begin{align}\label{eql}
\gamma_{xz} \sigma(\lambda(z), \lambda(y)) - \sigma(\lambda(z), \lambda(y)) \gamma_{yz} = 0 .
 \end{align}
 Since $ \sigma(i,j) \in K^* $, then $\gamma_{xz} = \gamma_{yz} $ for all $x < z$  and  $y < z$. Therefore, 
 $D$ is a transposed Jordan $*$-derivation.
\end{proof}
\begin{remark}\label{03}
Let    $D\in \JD^{\phi}$ be a transposed Jordan $*$-derivation with involution $\phi=M_\sigma \circ\hat{\lambda}$
and $R_{\mu}$ a right multiplication operator with $\mu$ an invertible element in $I(X,K)$.
By Lemma \ref{3} we have $D(e_{xx})=0$  for all $x\in X$. 
Let $\sigma_0(\lambda(y), \lambda(x)):=\sigma(\lambda(y), \lambda(x))\mu(y,y)^{-1}\mu(x,x)$, and it is well defined by
\cite[Theorem 1.2.3]{SpDo}. It is clear that $\sigma_0$ is a multiplicative. Following the proof of Lemma \ref{3}, we have
  $R_{\mu}\circ D$ is a transposed Jordan $*$-derivation.
\end{remark}
\begin{lemma}\label{4}
Let $\phi=M_\sigma \circ\hat{\lambda}$ be an involution of $I(X,K)$. 
Then each Jordan $*$-derivation of $I(X,K)$ is the sum of inner $*$-derivation and transposed Jordan $*$-derivation.
\end{lemma}
\begin{proof}
Let $D$ be a Jordan $*$-derivation of  $I(X,K)$. By Lemma  \ref{2},
there exist an inner $*$-derivation $\Delta_{f}^{\phi}$, 
such that $D(e_{xx})=\Delta_{f}^{\phi}(e_{xx})$ for all $x\in X$. 
Then we have $(D-\Delta_{f}^{\phi})(e_{xx})=0$, by  Lemma \ref{3}, we get $D-\Delta_{f}^{\phi}$ is a transposed Jordan $*$-derivation.
\end{proof}

\begin{theorem}\label{5}
Every Jordan $*$-derivation of $I(X,K)$ is an inner $*$-derivation and a transposed Jordan $*$-derivation.
\end{theorem}
\begin{proof}
Let $\theta$ be an involution of $I(X,K)$ and $D$ a Jordan $*$-derivation of  $I(X,K)$. By Lemma \ref{autin},  $\theta=\psi_{u}\circ M_{\sigma}\circ \hat{\lambda}$ for some invertible $\mu \in I(X,K)$, involution $\lambda$ of $X$ and multiplicative element $\sigma\in I(X,K)$, we denote $\phi=M_{\sigma}\circ\hat{\lambda}$ is an involution of $I(X,K)$ and it is clear that $\theta=\psi_{\mu }\circ\phi$. 

By Corollary \ref{c1}, we have $R_\mu\circ D$ be a Jordan $*$-derivation with involution $\phi$ where $R_\mu$ be a right multiplication operator. Using Lemma \ref{4}, $R_\mu\circ D$ is an inner $*$-derivation and a transposed Jordan $*$-derivation, i.e. $R_\mu\circ D=\Delta_{f}^{\phi}+D_{0}^{\phi}$. 
and then $$D=R_{\mu^{-1}}\circ R_\mu\circ D=R_{\mu^{-1}}\circ\Delta_{f}^{\phi}+R_{\mu^{-1}}\circ D_{0}^{\phi}$$
By Lemma \ref{c2}, $R_{\mu^{-1}}\circ \Delta_{f }^{\phi}= \Delta_{f\mu^{-1}}^{\theta}$ is an inner $*$-derivation and it is also a Jordan $*$-derivation.
This implies $R_{\mu^{-1}}\circ D_{0}^{\phi}=D-\Delta_{f\mu^{-1}}^{\theta}$ is a Jordan $*$-derivation. Since $D_{0}^{\phi}$ is a transposed Jordan $*$-derivation, by Remark \ref{03},  we have $R_{\mu^{-1}}\circ D_{0}^{\phi}$ is a Jordan $*$-derivation.  
\end{proof}


\end{document}